\documentclass[oneside,12pt]{amsart}
\usepackage[right=2cm]{geometry}
\usepackage{color}
\usepackage{amssymb}

\newcommand{\N}{\mathbb{N}}
\newcommand{\Nz}{\N_0}

\newcommand{\C}{\mathbb{C}}
\newcommand{\Z}{\mathbb{Z}}

\newcommand{\lp}{\left(}
\newcommand{\rp}{\right)}

\newcommand{\hh}{\hat{h}}
\newcommand{\hu}{\hat{u}}
\newcommand{\hv}{\hat{v}}

\newcommand{\hT}{\hat{T}}
\newcommand{\hU}{\hat{U}}
\newcommand{\hV}{\hat{V}}
\newcommand{\hPsi}{\hat{\Psi}}
\newcommand{\tT}{\tilde{T}}

\newcommand{\teta}{\tilde{\eta}}

\newcommand{\tp}{{\tilde{p}}}

\newcommand{\bz}{\bar{z}}

\newcommand{\tA}{\tilde{A}}

\renewcommand{\th}{\tilde{h}}

\newcommand{\hnu}{\hat{\nu}}

\newcommand{\hA}{\hat{A}}

\newcommand{\oJ}{\mathsf{J}}
\newcommand{\oU}{\mathsf{U}}
\newcommand{\oV}{\mathsf{V}}
\newcommand{\oX}{\mathsf{X}}

\newcommand{\hoU}{\hat{\oU}}
\newcommand{\hoV}{\hat{\oV}}
\newcommand{\hoX}{\hat{\oX}}

\newcommand{\cP}{\mathcal{P}}
\newcommand{\cQ}{\mathcal{Q}}
\newcommand{\hcP}{\hat{\cP}}

\newcommand{\cR}{\mathcal{R}}
\newcommand{\cU}{\mathcal{U}}

\newcommand{\tre}[1]{}

\newcommand{\fS}{\mathsf{S}}

\DeclareMathOperator{\ev}{\operatorname{ev}}
\DeclareMathOperator{\Wr}{\operatorname{Wr}}
\DeclareMathOperator{\Erf}{\operatorname{Erf}}

\newcommand{\Res}[1]{\operatorname*{Res}\limits_{#1}}

\newtheorem{thm}{Theorem}[section]
\newtheorem{prop}[thm]{Proposition}
\newtheorem{cor}[thm]{Corollary}
\newtheorem{lem}[thm]{Lemma}
\newtheorem{definition}[thm]{Definition}

\title{Exceptional Krall polynomials}
\author{A. Kasman}
\address{Department of
  Mathematics, College of Charleston}
\email{kasmana@cofc.edu}
\author{ R. Milson}
\address{Department of Mathematics and Statistics, Dalhousie
  University}
\email{rmilson@dal.ca}

\begin{document}
\begin{abstract}
  The purpose of this paper is to describe a novel class of
  exceptional Krall orthogonal polynomials of Hermite type.  This
  means that the polynomials in question are (i) orthogonal with
  respect to a Hermite-type weight; (ii) are the eigenfunctions of a
  higher-order differential operator; (iii) the degree sequence of the
  polynomial family in question is missing a finite number of degrees.
  Regarding the second point, unlike the known class of exceptional
  Hermite polynomials that satisfy a second-order eigenvalue equation,
  the polynomials we introduce here are not eigenfunctions of any 2nd
  order differential operator, but are for one of 4th order. Regarding
  the third point, our family does not include a polynomial of degree
  zero and consequently satisfies a 5th order recurrence relation
  instead of the classical 3-term relation.  For these reasons, we
  refer to the polynomials under discussion as exceptional Krall
  polynomials.
\end{abstract}
\maketitle
\section{Introduction}
Classical orthogonal polynomials of the Hermite, Laguerre, Jacobi and
Bessel type are precisely those families of orthogonal polynomials
that arise as eigenfunctions of a second-order differential
operator\footnote{The first 3 families arise as solutions of a
  classical Sturm-Liouville problem where as the Bessel family is
  based on a more general notion of orthogonality.}.  The problem
formulated by Krall \cite{K38}, is to find all families of orthogonal
polynomials that are eigenfunctions of a differential operator. Krall
also showed that the order of the eigenoperator is necessarily even
and went on to do the classification for 4th order eigenoperators.  In
this way of thinking, the classical families listed above correspond
to the subcase where the operator has order 2; this result is often
cited as the Bochner Theorem \cite{Bo29}.  Littlejohn \cite{L82} found
examples of orthogonal polynomials that satisfy a 6th order eigenvalue
relation and coined the term \textit{Krall orthogonal polynomials}.
By contrast to \textit{classical orthogonal polynomials}, such
families are defined as orthogonal eigenpolynomials of differential
operators of arbitrary order.

The question of classification followed quickly and was often cast
using the terminology of the \textit{Bochner-Krall class} or as the
\textit{Bochner-Krall problem} \cite{EKLW01,H00}.  Such questions
involve various notions of orthogonality, but the latest trends
seem to be in the direction of replacing orthogonality with respect to
a weight with the requirement that the eigenpolynomials
satisfy a recurrence relation.  In \cite{H00} one finds the notion of
an \textit{extended Bochner-Krall problem} which imposes the existence
of a 3-term recurrence relation.  A more general formulation is found
in \cite{HST25} where they speak of a \textit{generalized algebraic
  Bochner-Krall problem} and replace orthogonality with respect to a
weight with the existence of a recurrence relation of an arbitrary
order.

This way of framing the classification question has a natural
interpretation in terms of bispectrality where the eigenpolynomials
are at once the eigenfunctions of a differential operator in a
continuous variable and of a shift operator in a discrete variable.
The notion of bispectrality was introduced by Duistermaat and Gr\"unbaum
in \cite{DG87} who answered the following question:
\begin{quote}
  \textit{Duistermaat-Gr\"unbaum Problem: to determine all
    Schr\"odinger operators whose eigenfunctions are also
    eigenfunctions of a dual differential operator in the spectral
    variable.}
\end{quote}
To quote \cite{H00},
\begin{quote}
  \textit{The extended Bochner-Krall problem can be thought of as a
  partially discretized version of the Duistermaat-Gr\"unbaum problem,
  with the Schr\"odinger operator replaced by a doubly infinite
  tridiagonal matrix.}
\end{quote}
At present, the latter problem remains unsolved, but in a series of
papers in the late 90s, Gr\"unbaum, Haine, and Horozov showed that
known examples, including the examples found by Krall and Littlejohn,
may be constructed using the method of bispectral Darboux
transformations \cite{GH97,GHH97}.  The latter notion was introduced
simultaneously in \cite{BHY96} and \cite{KR97}; both papers exhibited
novel examples of rank $r$ bispectral problems, that is dual operators
with a common $r$-dimensional space of eigenfunctions.  Alternatively,
a general construction of bispectral Krall-Laguerre and Krall-Jacobi
polynomials based on the method of $\mathcal{D}$-operators is given in
\cite{DI22}.

In a parallel development, the Bochner Theorem was generalized by
considering orthogonal polynomials as solutions of a general
\textit{polynomial Sturm-Liouville problem} \cite{GKM09}. This
terminology refers to a Sturm-Liouville problem with polynomial
eigenfunctions but without the assumption that there be an
eigenpolynomial of every degree.  The resulting class of orthogonal
polynomials, now known as \textit{exceptional orthogonal polynomials},
has been intensively studied over the past 15 years with a complete
classification now within reach \cite{GGM13,GGM19,GGM24}.

A connection between exceptional polynomials and bispectral Darboux
transformations was discovered by studying the recurrence relations
enjoyed by exceptional Hermite polynomials \cite{GKKM16,KM20}. The
construction works by factorizing a certain $\ell$th degree polynomial
of $T_2$, the classical Hermite operator, as $A^\dag A=p(T_2)$ where
$A$ is a $\ell$th order operator and $A^\dag$ is its formal adjoint
relative to the Hermite weight $\exp(-x^2)$.  The exceptional operator
$\hT$ is then recovered from the factorization $p(\hT) = AA^\dag$.
This procedure works as long as we take
$p(T) = (T+2k_1)\cdots (T+2k_\ell)$ where $k_1,\cdots, k_\ell$ is a
choice of distinct natural numbers.  The corresponding recurrence
relations can then be recovered by applying the bispectral involution
to the corresponding point in the adelic Grassmannian and translating
the differential operators in the spectral variable into difference
operators.

The construction that we exhibit here is structurally similar in that
we exhibit the factorization $A^\dag A = p_2(T_2;b)$, where
\begin{equation}
  \label{eq:p2def}
p_2(T;b):=(T+b^2)(T+b^2-2)
\end{equation}
is the indicated 2nd degree polynomial that depends on a continuous
parameter $b$.  The Krall eigenoperator is then obtained by a Darboux
transformation as $\hT_4= A A^\dag$.  In this construction, it is not
possible to express the 4th-order $\hT_4$ as a polynomial of a second
order operator, and so the resulting eigenpolynomials may be
reasonably construed as exceptional analogues of the usual Krall
polynomials.  By contrast with orthogonal polynomials where every
degree is present, the exceptional Krall polynomials do not satisfy a
3-term recurrence relation, but rather a fifth order recurrence
relation.

Since the above construction belongs to the class of bispectral
Darboux transformations, there is a dual construction of the
recurrence relations based on a factorization of the $(\oJ+b)^2$
difference operator, where $\oJ$ is the Jacobi operator for the
classical Hermite polynomials; see \eqref{eq:Jdef} for the explicit
definition.  Factorizations of $(\oJ+b)^2$ were fully classified by
Gr\"unbaum, Haine, and Horozov \cite{GHH97}. They labelled the
resulting eigenfunctions \emph{Krall-Hermite polynomials} because
these polynomials satisfy both a 5-term recurrence relation and are
eigenfunctions of a 4th order operator.  Their construction produces a
2-parameter family of polynomials sequences.  These parameters are
called $\alpha$ and $k$ in Section 3 of \cite{GHH97}.  The resulting
polynomial sequences can be reasonably called exceptional because
there is no polynomial of degree 0, but cannot be called orthogonal
because Darboux transformations act on operators and do not have a
meaningful action on the inner product.  The authors of \cite{GHH97}
do not remark on the missing degree, but do say:
\begin{quote}
  \textit{These polynomials cannot be orthogonal, but as we remarked
    earlier, in some electrostatic problems it is only the
    differential equation that plays an important role.}
\end{quote}
Remarkably, orthogonality can be recovered by considering the special
case of $\alpha=1/k$.  With this specialization, the Krall-Hermite
polynomials of \cite{GHH97} coincide with the exceptional Krall
polynomials that we construct here. The relation between our parameter
and the parameters of \cite{GHH97} is $b=-1/\alpha=-k$.

The paper is organized as follows.  Section \ref{sect:mr} begins with
some necessary background on classical Hermite polynomials, and then
passes to the statement of the main definitions and results: the
exceptional Krall polynomials, their exponential generating function,
an eigenvalue equation, orthogonality, and finally, a recurrence
relation. Section \ref{sect:proofs} is devoted to the proofs of these
results.  Unlike the classical Hermite polynomials, the exceptional
Krall polynomials introduced here depend on a continuous parameter,
which we take to be $b$.  Certain of the properties described above
undergo various degenerations for certain characteristic values of
this parameter; this is described in Section \ref{sect:charvals}.
Finally, Section \ref{sect:bispectrality} is devoted to showing how
the exceptional Krall polynomials may be obtained from their classical
counter parts by means of a bispectral Darboux transformation.  As
such, the eigenoperator and the recurrence relation given in Section
\ref{sect:mr} can be extended to certain commutative algebras of
eigenoperators and recurrence relations.  The former is generated by a
4th order and a 6th order differential operator, while the latter is
generated by a 5th order and a 7th order recurrence relations.
Section \ref{sect:bispectrality} exhibits these higher order
generators and shows how their construction is related to the
methodology of bispectral Darboux transformations.
\section{Key definitions and main results.}
\label{sect:mr}
\subsection{Classical Hermite polynomials}
We begin by recalling the definition and essential properties of the
classical Hermite polynomials.  We will utilize the monic form of these
polynomials:
\begin{align}
  \label{eq:Hndef} 
  h_n(x) &=    \sum_{j=0}^{\lfloor n/2\rfloor} \frac{n!}{(n-2j)!
    j!} x^{n-2j} (-4)^{-j}
\end{align}
These can also be defined by means of the 3-term recurrence relation
\begin{equation}
  \label{eq:hermrel3}
  x h_n = h_{n+1} + \frac{n}{2} h_{n-1},\quad n=0,1,2,\ldots
\end{equation}
subject to initial condition $h_0=1$.  The $h_n(x),\; n\in \Nz$
constitute a family of orthogonal polynomials because they obey
orthogonality relations
\begin{equation}
  \label{eq:hermorthog}
  \frac{1}{\sqrt{\pi}}   \int_{-\infty}^\infty  h_m(x) h_n(x) e^{-x^2}
  dx = \delta_{mn}  \nu_n,\quad m,n\in \Nz, 
\end{equation}
where
\begin{equation}
  \label{eq:nuidef}
 \nu_n = 2^{-n}n!,\quad n\in \Nz.
\end{equation}
In addition to being orthogonal, the $h_n(x),\; n\in \Nz$ are
considered to be classical polynomials because they are polynomial
solutions of a second-order eigenvalue equation:
\begin{equation}
  \label{eq:hermde}
  T_2 h_n =- 2n h_n ,\quad n\in \Nz,
\end{equation}
where
\begin{equation}
  \label{eq:T2def}
  T_2(x,\partial_x) = \partial_x^2 -2x \partial _x.
\end{equation}
is the classical Hermite differential operator.

Classical Hermite polynomials obey a lowering and raising relations
\begin{align}
  \label{eq:hermlrel}
  &\partial_x h_n =  n h_{n-1},\quad n=1,2,\ldots.\\
  \label{eq:hermrrel}
  &\lp-\frac12  \partial_x+x\rp h_n = h_{n+1},\quad n\in \Nz.
\end{align}
These relations are, essentially, a restatement of the 3-term
recurrence relation \eqref{eq:hermrel3}.  The raising relation, in
particular, directly implies the Rodriguez formula
\begin{equation}
  \label{eq:hermrod}
  h_{n}(x) = (-2)^{-n} e^{x^2}\partial_x^n e^{-x^2},\quad n\in \Nz.
\end{equation}
\noindent
Hermite polynomials may also be recovered as the coefficients of the
generating function
\begin{equation}
  \label{eq:Psidef}
  \Psi(x,z) := \exp(x z - z^2/4) = \sum_{n=0}^\infty h_n(x) \frac{z^n}{n!}
\end{equation}
This can be demonstrated by observing that
\begin{equation}
  \label{eq:hermrrdiffrel}
  (z -2x +2\partial_z) \Psi_0 = 0.
\end{equation}
Expanding in powers of
$z$, we see that the latter is equivalent to the 3-term relation
\eqref{eq:hermrel3}.

Properties \eqref{eq:hermorthog}, \eqref{eq:hermde},
\eqref{eq:hermlrel}, and \eqref{eq:hermrrel} can all be established
directly by straightforward manipulations of the above $\Psi_0$.  For
example, the evident relations
\begin{align}
  \label{eq:T2Psiz}
  T_2(x,\partial_x) \Psi(x,z) &= -2z \partial_z
                                  \Psi(x,z)=(z^2-2xz)\Psi(x,z)\\ 
  x \Psi(x,z) &= \lp \partial_z + \frac{z}{2} \rp \Psi(x,z)
\end{align}
directly imply \eqref{eq:hermrel3} and \eqref{eq:hermde}.
Orthogonality \eqref{eq:hermorthog} can be recovered by observing that
\[ \frac{1}{\sqrt{\pi}}\int \Psi(x,z) \Psi(x,\bz) e^{-x^2}dx =
  \frac{1}{2} \exp(z\bz/2) \Erf(x-\Re z)+C(z), \] where
\[ \Erf(x) = \frac{2}{\sqrt{\pi}} \int_0^x e^{-u^2} du \] is the usual
error function. It follows that
\[ \frac{1}{\sqrt{\pi}}\int_{-\infty}^\infty \Psi(x,z) \Psi(x,\bz)
  e^{-x^2} dx = \exp(z\bz/2). \] The orthogonality relation for the
Hermite polynomials \eqref{eq:hermorthog} is then obtained by
expanding both sides of the above equation as power series in $z$ and
$\bar{z}$.

\subsection{The exceptional Krall polynomials}
We now define some novel analogues of the classical Hermite
polynomials.  An important note is that, unlike classical and
exceptional Hermite polynomials, the polynomials we consider here
depend on a continuous parameter, which we take to be $b$.  Our
primary definition is the following.
\begin{equation}
  \label{eq:H1def}
  \begin{aligned}
    \hh_n(x;b) &:= (b^2-2n) h_{n}(x)-
 \frac{b h_n(x)  + n    h_{n-1}(x)}{x+b},\quad n=0,1,\ldots
  \end{aligned}
\end{equation}
Setting
\begin{align}
  \label{eq:Uxdef}
  U(x,\partial_x;b)
  &=  \partial_x + b ,\\
  \label{eq:hAdef}
  \hA(x,\partial_x;b)
  &=  T_2+b^2 - (x+b)^{-1}U
\end{align}
we can  restate \eqref{eq:H1def} simply as
\begin{equation}
    \label{eq:hhA2H}
    \hh_n =  \hA h_n,\quad n=0,1,2,\ldots.
\end{equation}
We will refer to the resulting family of rational functions as being
\emph{quasi-polynomial.}  This means that this family consists of
polynomials divided by a common denominator; in this case the
denominator is $x+b$.  The exceptional Krall polynomials
$\th_{n+1}(x),\; n=0,1,\ldots, $ can then be given as the numerators of
this family, namely as
\begin{equation}
  \label{eq:tHdef}
  \begin{aligned}
    \th_{n+1}(x;b) &=  (x+b) \hh_{n}(x;b),\quad
    n=0,1,2,\ldots\\
    &= (b^2-2n)(x+b)h_n(x)-b h_n(x) - n h_{n-1}(x)
  \end{aligned}
\end{equation}
The resulting expressions are polynomials in $x$ having the indicated
degree with a quadratic dependence on $b$.

By inspection, if $b^2\ne 2n$, then $\deg \th_{n+1}(x) = n+1$. Thus,
generically, $\th_1,\th_2,\th_3,\ldots$ span a codimension-1
polynomial subspace, and as we demonstrate below, may be regarded as a
new class of orthogonal polynomial that are somehow related to the
classical Hermite polynomials.

\subsection{Key properties}
Next, we exhibit a suite of properties of the above defined polynomial
sequence.  These properties are all analogues of the properties
enjoyed by the classical Hermite polynomials: an exponential
generating function, an eigenvalue equation, orthogonality, and a
recurrence relation.  The proofs of these assertions are given in the
sequel.  The eigenvalue equation involves a fourth order operator,
justifying the ``Krall'' designation.  Since there are missing
degrees, the recurrence relation is fifth order.  This is another
consequence of the ``exceptionality'' of this family.
\subsubsection{The generating function}
Applying the intertwining operator to the classical generating
function \eqref{eq:Psidef} gives a generating function
for the  quasi-polynomials \eqref{eq:H1def}. Set
\begin{equation}
  \label{eq:hPsidef}
  \hPsi(x,z;b) = 
  \lp z^2-2zx + b^2 - \frac{z+b}{x+b} \rp \Psi(x,z)
\end{equation}
\begin{prop}
  \label{prop:genfunc}
  We have
  \begin{equation}
  \hPsi(x,z;b) = \sum_{n=0}^\infty \hh_{n}(x)
  \frac{z^n}{n!}  
\end{equation}
\end{prop}
\subsubsection{The Krall property}
Next, we exhibit the fourth-order exceptional analogue of the Hermite
differential equation \eqref{eq:hermde}.  
\begin{prop}
  \label{prop:tTeigen}
  The exceptional polynomials $\th_{n+1}(x),\; n=0,1,2,\ldots $ enjoy
  the following eigenvalue relation:
  \begin{equation}
  \label{eq:tTeigen}
      \tT_4 \th_{n+1} = p_2(-2n)\, \th_{n+1}= (b^2-2n)(b^2+2-2n)
      \th_{n+1},\quad n=0,1,2,\ldots 
    \end{equation}
    where
\begin{align}
    \label{eq:tT4}
  \tT_4
  &= 4(T_2+b^2-1) \hA-3 (T_2+b^2) \lp  T_2 +b^2-2\rp\\\nonumber 
  &=  \partial_x^4-4\lp x+ \frac1{x+b}\rp \partial_x^3 + \ldots
\end{align}
and where $p_2$ is the polynomial defined in \eqref{eq:p2def}.
\end{prop}
\noindent A similar eigenvalue relation holds for the
quasi-polynomials $\hh_n(x),\; n\in \Nz$ as defined in
\eqref{eq:H1def}.  Let $T_2,U,\hA$ be as in \eqref{eq:T2def}
\eqref{eq:Uxdef} \eqref{eq:hAdef}, and set\footnote{ The ``adjoint''
  nomenclature will be justified in the next section.}
\begin{align}
  \label{eq:Vxdef}
  V(x,\partial_x;b)&=U^\dag= -\partial_x + 2x+b\\
  \label{eq:hAadjdef}
  \hA^\dag(x,\partial_x)
  &= T_2 + b^2-V\circ (x+b)^{-1},\\
%   &=T_2 + b^2-2+(x+b)^{-1}U-(x+b)^{-2},\\
  \label{eq:hT4}
  \hT_4 &=\hA \hA^\dag\\ \nonumber
         &= (T_2+b^2)(T_2+b^2-2)\\ \nonumber
         &\qquad + 4(x+b)^{-1}\lp2\partial_x+b\rp
           - 4(x+b)^{-2}\lp \partial_x^2+2b \partial_x+b^2+1\rp
  \\\nonumber
         &\qquad + 8(x+b)^{-3}\lp \partial_x + b\rp -8 (x+b)^{-4}
\end{align}
\begin{prop}
  We have
  \label{prop:hTeigen}
  \begin{equation}
      \label{eq:hTeigen}
    \hT_4 \hh_n     = (b^2-2n)(b^2+2-2n)\, \hh_n, \quad  n=0,1,2,\ldots
  \end{equation}
\end{prop}
\noindent
% This relation is demystified once we assert that $\hT_4$ arises from
% the following intertwining relation.
% \begin{equation}
%   \label{eq:T4intertwine}
%   \hT_4 \circ \hA = \hA \circ \pi_2(T_2)
% \end{equation}
% where
% \begin{align*}
%   \pi_2(T)& = \frac14 (T-2)(T+2b^2).
% \end{align*}
% \tre{
%   Indeed, observe that
% \[
%   \begin{aligned}
%   \pi_2(T_2) h_{i} &= \pi_2(-2n)h_{i}\\
%   &= \frac14(-2n-2)(-2n+2b^2)h_{i}\\
%   &=  (n+1)(i-b^2)h_{i},\quad n=0,1,2,\ldots.
% \end{aligned}
% \]}
% The eigenvalue relation \eqref{eq:hTeigen} now follows directly from
% \eqref{eq:hhA2H}.
\subsubsection{Orthogonality}
\begin{prop}
  \label{prop:hHorthog}
  The exceptional Krall polynomials satisfy the orthogonality relations
  \begin{equation}
    \label{eq:thorthog}
    \frac{1}{\sqrt{\pi}} \int_{-\infty}^\infty \th_{n+1}(x;b)
    \th_{m+1}(x;b)\frac{e^{-x^2}}{(x+b)^2}  
    dx = \delta_{mn} p_2(-2n)\nu_{n},\quad m,n=0,1,2,\ldots,
  \end{equation}
  where the integral is taken along \textbf{any} contour from $-\infty$ to
  $+\infty$ that avoids the singularity at $x=-b$.  Equivalently, the
  quasi-polynomials $\hh_n(x)$ satisfy
  \begin{equation}
    \label{eq:hHintorthog}
    \frac{1}{\sqrt{\pi}} \int_{-\infty}^\infty \hh_n(x) \hh_m(x)e^{-x^2}
    dx = \delta_{mn} p_2(-2n)\nu_{n},\quad m,n=0,1,2,\ldots
  \end{equation}
\end{prop}
\noindent
These kinds of integrals were used to describe generalized
orthogonality for exceptional Hermite polynomials in \cite{HHV16}.  By
inspection, the resulting inner product is not positive definite.
This fact is related to the presence of a singularity on the real
line.  See the just cited paper has some comments and citations
regarding this matter.
\subsubsection{Recurrence relation.}
Let $\fS_n$ denote the shift operator in the variable $n$, so that for
a sequence $f_n,\; n\in \Z$ we have
\begin{equation}
  \label{eq:fSdef}
  \fS_n f_n = f_{n+1},\quad \fS_n^kf_n = f_{n+k},\; k\in \Z.
\end{equation}
Define the difference operator
\begin{equation}
  \label{eq:Jdef}
 \oJ(n,\fS_n) = \fS_n + \frac{n}{2} \fS_n^{-1},\quad \oJ f_n =
  f_{n+1} + \frac{n}2 f_{n-1}.
\end{equation}
% so that, for a sequence $f_n,\; n\in \Nz$, we have
% \[ (\oJ f)_n = f_{n+1} + \frac{i}{2}\, f_{i-1}. \]
Also define the difference operators
%$\tU(n,\fS_n,b),U(n,\fS_n,b), \tV(n,\fS_n,b), V(n,\fS_n,b)$
\begin{equation}
  \label{eq:UVdef}
  \begin{aligned}
    \oU(n,\fS_n;b) &= b+ n \fS_n^{-1},& \hoU(n,\fS_n;b)&=(b^2-2n)^{-1} \oU\\
    \oV(n,\fS_n;b) &= 2\fS_n + b,& \hoV(n,\fS_n;b)&= (b^2-2n-2)^{-1} \oV.
  \end{aligned}
\end{equation}
\begin{prop}
    \label{prop:rr5}
    The  quasi-polynomials $\hh_n(x),\; n\in \Nz$ satisfy the
    following 5th order recurrence relation.
  \begin{equation}
    \small
    \label{eq:rr5}
    \begin{aligned}
      (x+b)^2 \hh_n &= (\oJ+b)^2\hh_n+ \lp \oV \hoV- \oU \hoU\rp \hh_n \\
      &= \hh_{n+2}+ 2 b \hh_{n+1}+ \lp b^2+n-\frac12\rp \hh_n
      + bn\hh_{n-1}+\frac14 n(n-1) \hh_{n-2}\\
      &\qquad + 2\hv_{n+1}+b\hv_n + b\hu_n + n \hu_{n-1}
    \end{aligned}
  \end{equation}
  where
  \begin{align}
    \label{eq:hvdef}
    \hv_n &= \hoV\hh_n = \frac{2\hh_{n+1}+b \hh_{n}}{b^2-2n-2}\\
    \label{eq:hudef}
    \hu_n &= \hoU\hh_n = \frac{b\hh_n+n \hh_{n-1}}{b^2-2n}.
  \end{align}
\end{prop}
\noindent
Note that
the above defined $\oJ$ serves as the Jacobi operator for the monic
Hermite polynomials in the sense that the classical recurrence
relation \eqref{eq:hermrel3} may be rewritten as the eigenvalue relation
\[ \oJ h_n = x h_n.\]
As a consequence,
\begin{align*}
  (x+b)^2 h_n
  &   = (\oJ+b)^2 h_n 
  % &=
  %   h_{i+2}+ 2 b h_{n+1}+ \lp b^2+i+\frac12\rp h_n + bi
  %   h_{i-1}+ \frac14 i(i-1) h_{i-2}  
\end{align*}
Thus, \eqref{eq:rr5} may be justly regarded as a rational modification
of the above classical Hermite recurrence relation.  The recurrence
relation for the exceptional polynomials $\th_{n+1}(x),\; n=0,1,2,\ldots$
has the same form; one merely needs to multiply both sides of
\eqref{eq:rr5} by $(x+b)$ and conjugate the recurrence relation by a
shift operator in the index $n$.

\section{The proofs}
\label{sect:proofs}
We begin by justifying the form of the exceptional generating function.
\begin{proof}[Proof of Proposition \ref{prop:genfunc}]
By inspection,
\begin{align*}
  \partial_x \Psi&= z \Psi\\
  U \Psi &= (z+b)\Psi.
\end{align*}
Consequently, by \eqref{eq:T2Psiz},
\begin{equation}
  \label{eq:hPsidef2}
  \begin{aligned}
    \hA(x,\partial_x,b) \Psi(x,z)
    & = \lp T_2(x,\partial_x) + b^2 - \frac{z+b}{x+b}\rp
    \Psi(x,z)\\
    &= \hPsi(x,z;b)
  \end{aligned}
\end{equation}
  The desired conclusion now follows by \eqref{eq:Psidef} and
  \eqref{eq:hhA2H} .
\end{proof}
\noindent
The proof of the eigenvalue property relies on the following Lemmas.
\begin{lem}
  \label{lem:A2conj}
  Let $T_2,U,V$ be as defined in   \eqref{eq:T2def}, \eqref{eq:Uxdef}, \eqref{eq:Vxdef}.  We have
  \begin{align}
    \label{eq:T2UV}
    T_2+b^2
    &= -V U +2 b(x+b)\\
    \label{eq:UVJ}
    U+V
    &= 2(x+b)\\
    \label{eq:UVbrak}
    [U,V]
    &= 2\\
    \label{eq:Uconj}
    U \circ (x+b)^{-1}
    &=  (x+b)^{-1}U -(x+b)^{-2}\\
    \label{eq:Vconj}
    V \circ (x+b)^{-1}
    &=  (x+b)^{-1}V +(x+b)^{-2}\\
    \label{eq:A2conj}
    \hA \circ (x+b)
    &= (x+b)\hA^\dag\\
     \label{eq:T2conj}
    T_2\circ (x+b) &= (x+b) T_2 +U-V 
  \end{align}
\end{lem}
\begin{proof}
  These identities follow from definitions and straight-forward
  calculations.
\end{proof}
\begin{lem}
  Let $p_2(T;b)$ be the second-degree polynomial defined in
  \eqref{eq:p2def}. We then have
  \begin{equation}
    \label{eq:pT02}
    \hA^\dag \hA = p_2(T_2).
  \end{equation}
\end{lem}
\begin{proof}
  Using \eqref{eq:hAdef} \eqref{eq:hAadjdef} \eqref{eq:T2UV}
  \eqref{eq:UVJ} \eqref{eq:UVbrak}, we have
  \begin{align*}
    \hA^\dag \hA
    &= (T^2+b^2-V(x+b)^{-1})(T^2+b^2-(x+b)^{-1} U)\\
    &=(T^2+b^2)^2 +V(x+b)^{-1}VU+VU(x+b)^{-1}V-4b(x+b)
      + V (x+b)^{-2} U
  \end{align*}
  By \eqref{eq:Uconj},
  \begin{align*}
    &V(x+b)^{-1}VU+UV(x+b)^{-1}V= V(x+b)^{-1}VU+V(x+b)^{-1}UV-
      V(x+b)^{-2} U\\
    &=2VU - V(x+b)^{-2} U.
  \end{align*}
  Hence, by \eqref{eq:T2UV},
  \begin{align*}
    \hA^\dag \hA
    &= (T^2+b^2)^2+2VU-4b(x+b)\\
    &=(T^2+b^2)^2- 2(T^2+b^2) = (T^2+b^2)(T^2+b^2-2).
  \end{align*}
\end{proof}
\begin{proof}[Proof of Proposition \ref{prop:hTeigen}]
  Identity \eqref{eq:pT02} implies the following  intertwining
  relation:
  \[ \hT \hA = \hA\, p_2(T_2) .\] The desired conclusion now follows by
  \eqref{eq:hhA2H} and because by \eqref{eq:hermde} we have
  \[ p_2(T_2) h_n = p_2(-2n) h_n,\quad n\in \Nz.\]
\end{proof}
\noindent
\begin{proof}[Proof of Propostion \ref{prop:tTeigen}]
  Since $\th_{n+1}(x) = (x+b) \hh_n(x)$, it suffices to show that
  \begin{equation}
    \label{eq:hcT4tcT4}
    (x+b)\hT_4 = \tT_4 \circ (x+b) .
  \end{equation}
  By \eqref{eq:tT4} and \eqref{eq:pT02},
  \[ \tT_4 = -3 \hA^\dag \hA+ 4(T_2+b^2-1)\hA .\]
  Hence by \eqref{eq:A2conj},
  \begin{align*}
    \tT_4 \circ (x+b) = -3 \hA^\dag\circ (x+b)\circ \hA^\dag
    +4(T_2+b^2-1)\circ (x+b)\circ \hA^\dag.
  \end{align*}
  By the definition \eqref{eq:hAadjdef},
  \begin{align*}
    &\hA^\dag \circ (x+b) = (T_2+b^2)\circ(x+b)-V
  \end{align*}
  Hence,   by \eqref{eq:T2conj} and \eqref{eq:UVJ},
  \begin{align*}
    \tT_4 \circ (x+b)
    &=\lp -3 \hA^\dag\circ (x+b)
    +4(T_2+b^2-1)\circ (x+b)\rp \circ \hA^\dag\\
    &=\lp (T_2+b)\circ (x+b)+3V-4(x+b))\rp \circ \hA^\dag\\
    &=\lp (x+b)(T_2+b)+U-V+3V-2U-2V)\rp \circ \hA^\dag\\
    &= (x+b)\hA \hA^\dag,
  \end{align*}
as was to be shown.
\end{proof}
\subsection{ Proof of Orthogonality}
\label{sec:genorthog}
In this section, we prove the quasi-polynomial orthogonality relations
\eqref{eq:hHintorthog}.  From that, the orthogonality of the
exceptional polynomials shown in \eqref{eq:thorthog} follows
immediately.  The first step is to make sense of the quadratic form
defined by the seemingly singular integral in \eqref{eq:hHintorthog}.
We do this by introducing a bilinear form on a certain vector space of
quasi-polynomials, and then showing that this quadratic form can be
realized by the integrals in question.

Let $\cP$ denote the vector space of complex-valued polynomials.  We
define a bilinear form on $\cP$ using the following 
definition:
\begin{equation}
  \label{eq:etadef}
   \eta(p_1,p_2) = \frac{1}{\sqrt{\pi}} \int_{-\infty}^\infty p_1(x)
  p_2(x) e^{-x^2} dx ,\quad p_1,p_2\in \cP.
\end{equation}
The above integral is path-independent and
gives the same value for every contour from $x=-\infty$ to
$x=+\infty$.  This bilinear form underlies the orthogonality of the
classical Hermite polynomials and has the following characterization
in terms of formal anti-derivatives.
% \footnote{If one evaluates the
%   indefinite integral of $\int p_1(x) p_2(x) e^{-x^2}dx$ in
%   Mathematica, the output will be a sum of a quasi-rational term in
%   $\hcP$ plus the expression
%   $\eta(p_1,p_2) \frac{\sqrt{\pi}}{2} \Erf(x)$.  Thus, one can recover
%   $\eta(p_1,p_2)$ by asking for the coefficient of $\Erf(x)$ times
%   $2/\sqrt{\pi}$.}.

Going forward, we will say that an expression $W(x)$ is
\emph{quasi-rational} provided the log derivative $W'(x)/W(x)$ is a
rational function.  Let
\[ \hcP(x) = \cP(x)     e^{-x^2} = \{ p(x) e^{-x^2}: p \in \cP \} \]
denote the vector space of the indicated quasi-rational functions.
\begin{prop}
  \label{prop:etaP}
  For every $p_1,p_2\in \cP$ the value $\eta(p_1,p_2)$ defined by
  \eqref{eq:etadef} is the unique constant such that
  \[ \int \lp p_1(x)p_2(x)- \eta(p_1,p_2)\rp e^{-x^2} dx \]
  defines a quasi-rational anti-derivative.  That anti-derivative is
  unique and belongs to $\hcP(x)$.
\end{prop}
\noindent
The next step is to extend this bilinear form to the quasi-polynomials
defined in \eqref{eq:H1def}.  For a rational function $f(x)$, let
\[  f(x) \hcP(x) =\{ p(x)f(x)
  e^{-x^2}\colon p\in \cP \},\quad n\in \Nz
\] be the vector space of consisting of quasi-rational expressions of
the indicated form.  Observe that the derivatives of elements of
$(x+b)^{-1} \hcP(x)$ (i) belong to $(x+b)^{-2}\hcP(x)$, and (ii) have
a vanishing residue at $x=-b$.  Conversely, the following holds.
\begin{prop}
  A quasi-rational $\phi(x)\in (x+b)^{-2} \hcP(x)$ has a single-valued
  anti-derivative if and only if
  \[ \Res{x=-b} \phi(x)=0. \] In such a case, there exists a unique
  anti-derivative that has the form of  a quasi-rational
  expression in $(x+b)^{-1}\hcP(x)$ plus a constant times $\Erf(x)$.
\end{prop}
\noindent
\begin{definition}
  Let $U(x,\partial_x)=\partial_x+b$ be the differential operator
  defined in \eqref{eq:Uxdef}. For $b\in \C$, let
  $\gamma_b= \ev_{-b}\circ U$ be the functional on $\cP$ defined by
  \begin{equation}
    \label{eq:gammabdef}
    \gamma_b(p) = (U p)(-b)=  p'(-b) + b p(-b),\quad
    p \in \cP.
  \end{equation}
  Let
  \begin{equation}
    \label{eq:Ubdef}
    \cU_b(x) = (x+b)^{-1} \ker \gamma_b = \{ (x+b)^{-1}p(x): p\in
    \cP,\; p'(-b) + b p(-b) = 0 \},\quad b\in \C
  \end{equation}
  be the subspace of quasi-polynomials in $(x+b)^{-1}\cP(x)$ whose
  numerator $p(x)$ satisfies the 1-point condition $\gamma_b(p) = 0$.
\end{definition}

\tre{\textbf{Internal remark.}
  \begin{equation}
    \label{eq:gammabdef}
    \gamma_b(p) = -\Res{x=-b} \hA p(x)
  \end{equation}}

% \begin{prop}
%   The vector space $\cU_b$ can be characterized as 
% \end{prop}
% \begin{proof}
%   Let $q\in \cU_b$ so that, by assumption, $(\hA^\dag q)(x)$ has no
%   poles. By inspection of \eqref{eq:hAadjdef}, necessarily $q(x)$ must
%   be non-singular for $x\ne b$. The singular part of $\hA^\dag$ is
%   \[ \hU - (x+b)^{-2} = (x+b)^{-1}\partial_x-(x+b)^{-2} +
%     b(x+b)^{-1}.\] Hence, necessarily $q(x) =(x+b)^{-1}p(x),\; p\in \cP$; a
%   higher-order pole would produce a singularity. By \eqref{eq:A2conj},
%   we have
%   \[ \hA^\dag \lp (x+b)^{-1}p\rp = (x+b)^{-1}\hA p.\] Hence, it
%   suffices to show that $\hUp$ is a polynomial.  That is equivalent
%   the condition that $p'(-b) + bp(b) = 0$.
% \end{proof}
\noindent
The following characterization of $\cU_b$ is crucial to all that
follows.
\begin{prop}
  \label{prop:Ubres}
  Let $q(x)\in (x+b)^{-1} \cP(x)$. Then, $q\in \cU_b$ if and only if
  \[ \Res{x=-b} \lp q(x)^2 e^{-x^2}\rp=0.\]
\end{prop}
\begin{proof}
  Write $q(x) = (x+b)^{-1} p(x)$ and let $a_0 = p(-b)$ and
  $a_1 = p'(-b)$ so that
  \[ q(x) = \frac{a_{0}}{x+b}+ a_1   + O(x+b).\]
  Hence, 
  \begin{align*}
    q(x)^2 e^{-x^2}
    &\equiv \lp \frac{a_{0}}{x+b}+ a_1 \rp ^2 (1+2b(b+x))e^{-b^2}+ O(1)\\
    &\equiv \lp\frac{a_{0}^2}{(x+b)^2} + \frac{2a_{0}(a_{0} b +
      a_1)}{x+b}\rp e^{-b^2} + O(1)
  \end{align*}
  Hence, the residue in question vanishes if and only if $a_1+b a_0=0$.
\end{proof}

\begin{prop}
  \label{prop:etaUb}
  For every $q_1,q_2\in \cU_b$ there exists a unique constant
  $\teta_b(q_1,q_2)$ such that\tre{\footnote{If one submits such an evaluation
    to Mathematica, the output will be a sum of a quasi-rational term
    in $\hcP$ plus the expression
    $\eta(q_1,q_2) \frac{\sqrt{\pi}}{2} \Erf(x)$.  Thus, one can
    recover $\eta(q_1,q_2)$ by asking for the coefficient of $\Erf(x)$
    times $2/\sqrt{\pi}$.}}
  \[ \int \lp q_1(x)q_2(x)-    \teta_b(q_1,q_2)\rp e^{-x^2} dx \in
    (x+b)^{-1}\hcP(x).\] The 
  correspondence $(q_1,q_2) \mapsto \teta_b(q_1,q_2)$ defines a bilinear
  form on $\cU_b$.
\end{prop}
\begin{proof}
  Without loss of generality, assume that $q_1 = q_2$.  Once we have
  established the existence of the quadratic form $\teta_b(q,q)$, the
  case of the bilinear form follows by a polarization argument that
  expands $\teta_b(q_1+q_2,q_1+q_2)$.

  Suppose that $q\in \cU_b$.  By the argument from the proof of
  Proposition \ref{prop:Ubres}, if $q\in\cU_b$, then
  \[ q(x) = \frac{a_0}{x+b} - b a_0 + (x+b) p(x),\quad p \in \cP.\]
  Hence,
  \begin{align*}
    q(x)^2 e^{-x^2}
    &= a_0^2\lp\frac{1}{(x+b)^2} - \frac{2b}{x+b} \rp e^{-x^2}+r(x) e^{-x^2}\\
    &=      -a_0^2\,\partial_x \lp \frac{e^{-x^2}}{x+b}  \rp
    +   r(x) e^{-x^2},
  \end{align*}
  where $r\in \cP$ is a polynomial.   Expand $r(x)$ in Hermite
  polynomials.  The desired constant is then a multiple of the
  coefficient of $h_0(x)$.
\end{proof}

\begin{cor}
  \label{cor:qorthog}
  Let $q_1,q_2\in \cU_b$. Then, for every contour from $x=-\infty$ to
  $x=\infty$ that avoids the singularity, we have
  \[ \teta_b(q_1,q_2) = \frac1{\sqrt{\pi}}\int_{-\infty}^{\infty}
      q_1(x)q_2(x)e^{-x^2} dx .\]
\end{cor}
\noindent
In this sense, we have a generalization of \eqref{eq:etadef} that
extends the integral inner product to quasi-polynomials in $\cU_b$.
If there is no risk of ambiguity, going forward we will simply write
$\eta$ instead of $\teta_b$ since the formally defined bilinear form
and the bilinear form defined by an integral are equivalent.

Next, we justify the ``adjoint'' nomenclature used in definition
\eqref{eq:hAadjdef}.
\begin{lem}
  \label{lem:A2adj}
  The operator $\hA$ maps $\cP$ to $\cU_b$, while $\hA^\dag$ maps
  $\cU_b$ to $\cP$.  Moreover, the following adjoint relation holds:
  \begin{align}
    \label{eq:A2adj}
    \eta(\hA f,g) &= \eta(f,\hA^\dag g),\quad f\in \cP,\; g\in \cU_b.
  \end{align}
\end{lem}
\begin{proof}
  Let $U(x,\partial_x;b),V(x,\partial_x;b)$ be the differential
  operators as defined in \eqref{eq:Uxdef} \eqref{eq:Vxdef}.  The
  following identity can be verified by direct calculation:
  \[ (Uf(x))g(x) -f(x) (Vg(x)) = \partial_x \lp f(x) g(x)
    e^{-x^2}\rp\] In this sense $V=U^\dag$ is the adjoint of $U$.  In
  the same sense $T_2^\dag = T_2$ and
  $((x+b)^{-1}U)^\dag = V\circ (x+b)^{-1}$.  The justification amounts
  to the following identities:
  \begin{align}
    \label{eq:T2adj2}
    & (T_2 f)(x) g(x) e^{-x^2} - (T_2 g)(x) f(x) e^{-x^2} = \\ \nonumber
    &\qquad\qquad\partial_x\lp \lp f'(x) g(x) - f(x) g'(x)\rp
 e^{-x^2}\rp\\
    &((x+b)^{-1}U f(x)) g(x) e^{-x^2} - (V ((x+b)^{-1} g(x)) f(x) e^{-x^2} =\\
    \nonumber &\qquad\qquad
    \partial_x\lp \frac{f(x)g(x)e^{-x^2}}{x+b} \rp
  \end{align}
  By the definition
  \eqref{eq:hAadjdef}, $\hA^\dag =  T_2 +b^2 -((x+b)^{-1}U)^\dag$.  Hence,
  \[ \int\lp (\hA f)(x) g(x)  - f(x) (\hA^\dag g)(x)  \rp e^{-x^2}
    dx \in (x+b)^{-1}\hcP(x). \]  The conclusion now follows by Propositions
  \ref{prop:etaP} and \ref{prop:etaUb}.
\end{proof}

\begin{proof}[Proof of Proposition \ref{prop:hHorthog}]
  We claim that the  quasi-polynomials
  $\hh_n(x;b),\; n\in \Nz$ belong to $\cU_b(x)$ and
  satisfy the orthogonality relations
  \begin{equation}
    \label{eq:hHorthog}
    \eta(\hh_m,\hh_n) = \delta_{mn} p_2(-2n)\nu_n,\quad m,n=0,1,2,\ldots.
  \end{equation}
  with $p_2$ as defined in \eqref{eq:p2def} and $\nu_n$ as
  defined in \eqref{eq:nuidef}.   The claim \eqref{eq:hHintorthog}
  then follows by  Corollary \ref{cor:qorthog}.

  As before, let $p_2(T)= (T-b^2)(T+2-b^2)$. It suffices to observe that
  \begin{align*}
    \eta(\hh_n,\hh_m)
    &= \eta(\hA h_n, \hA h_m)\\
    &= \eta(\hA^\dag \hA h_n, h_m)\\
    &= \eta(p_2(T_2) h_n, h_m)\\
    &= p_2(-2n) \eta(h_n,h_m).
  \end{align*}
\end{proof}
\subsection{Proof of the recurrence relations.}
A family of orthogonal polynomials with every degree must satisfy a
3-term recurrence relation.  This is a consequence of the symmetry of
a multiplication operator, in the sense that
\begin{equation}
  \label{eq:xsym}
  \eta( x h_n , h_m) = \eta( h_n, x h_m),\quad n,m \in \Nz.
\end{equation}
Since $h_n(x)$ is orthogonal to every polynomial of lower degree, it
automatically follows that $x h_n(x),\; n\ge 2$ is orthogonal to
$h_0(x),\ldots, h_{n-2}(x)$.  A similar orthogonality argument can be
applied to establish the recurrence relation for the exceptional Krall
polynomials.

% Generically $\hh_n(x),\; n=0,1,2,\ldots$ span $\cU_b(x)$, as defined
% in \eqref{eq:Ubdef}.  Since $(x+b)^2\th_{n+1}(x),\; n\in \Nz$ is
% annihilated by the functional that defines $\cU_b(x)$, it follows that
% $(x+b)^2 \hh_n(x)$ belongs to $\cU_b(x)$.  However, by the symmetry of
% the multiplication operator $(x+b)^2$, it must be orthogonal to all
% $\hh_j(x)$ with $j\le i-3$ and all $\hh_j(x)$ with $j\ge i+3$. It only
% remains to justify the form of the coefficients shown in
% \eqref{eq:rr5}.

% We give two proofs; one is based on the  technique of orthogonality
% that can be employed to establish the classical 3-term recurrence
% relation.  The other proof utilizes generating functions.

\begin{proof}[Proof of Proposition \ref{prop:rr5}.]
Use the classical relation \eqref{eq:hermrel3} to rewrite
\eqref{eq:tHdef} as
\begin{equation}
  \label{eq:x+bhHi}
  \small
  (x+b)\hh_n(x;b)= A^{(1)}_1(n;b)h_{n+1}(x) + A^{(1)}_0(n;b) h_n(x) +
  A^{(1)}_{-1}(n;b) h_{i-1}(x),
\end{equation}
where
\begin{equation}
  \label{eq:Andef}
\begin{aligned}
  A^{(1)}_1(n;b) &= b^2-2n\\
  A^{(1)}_0(n;b) &= b(b^2-2n-1)\\
  A^{(1)}_{-1}(n;b) &= \frac{n}{2}(b^2-2n-2)
\end{aligned}
\end{equation}
Hence
\[ \eta(\hh_n,\hh_{m}) = 0 \] if $|n-m|>2$.  Generically
$\hh_n(x;b),\; n=0,1,2,\ldots$ span $\cU_b(x)$, as defined in
\eqref{eq:Ubdef}.  Observe that $(x+b)^2\th_{n+1}(x),\; n\in \Nz$ is
annihilated by $\gamma_b$.  This follows directly from the definition
of this functional in \eqref{eq:gammabdef}. It follows that
$(x+b)^2 \hh_n(x)$ belongs to $\cU_b(x)$.  Orthogonality immediately
implies that
\begin{equation}
  \label{eq:akjdef}
  (x+b)^2 \hh_n = \sum_{k=-2}^2 a^{(2)}_k(n;b) \hh_{n+k},
\end{equation}
where the coefficients $a^{(2)}_j(n;b)$ are some rational functions.
Taking inner-products of both sides of \eqref{eq:rr5} with
$\hh_{n+k},\; k\in \{-2,-1,0,1,2\}$ and using the fact that
\[ \eta((x+b)\hh_n,(x+b)\hh_{n+k}) = \eta((x+b)^2 \hh_n,
  \hh_{n+k}) \]
gives, respectively,
\begin{align*}
  a^{(2)}_{-2}(n)\,\hnu_{n-2} &=  A_1(n-2) A_{-1}(n) \nu_{n-1}\\
  a^{(2)}_{-1}(n)\,\hnu_{n-1} &=  A_0(n-1) A_{-1}(n) \nu_{n-1}+
                             A_1(n-1)A_0(n)\nu_n\\
  a^{(2)}_0(n)\,\hnu_n  &= A_{-1}(n)^2 \nu_{n-1} + A_0(n)^2 \nu_n + A_1(n)^2
                       \nu_{n+1} \\
  a^{(2)}_1(n)\,\hnu_{n+1} &= A_{-1}(n+1)A_0(n)\nu_n
                       +A_0(n+1) A_{1}(n) \nu_{n+1}\\
   a^{(2)}_2(n)\,\hnu_{n+2} &= A_{-1}(n+2) A_{1}(n) \nu_{n+1}
\end{align*}
with $\nu_n$ as defined in \eqref{eq:nuidef} and with
\begin{equation}
  \label{eq:hnudef}
  \hnu_n(b)= p_2(-2n)  \nu_n=(b^2-2n)(b^2-2n-2)2^{-i}n!.
\end{equation}
Applying the definitions \eqref{eq:Andef} gives
\begin{align*}
  a^{(2)}_{-2}(n)
  &= \frac14 n(n-1)\frac{b^2-2n-2}{b^2-2n+2}\\
  &= \frac14 n(n-1)\lp 1- \frac{4}{b^2-2n+2}\rp\\
  a^{(2)}_{-1}(n)
  &= \frac{bn \lp 
    (b^2-2n+1)(b^2-2n-2) +
    (b^2-2n-1)(b^2-2n+2)\rp}{2(b^2-2n+2)(b^2-2n)} \\
  &= bn \lp 1- \frac{1}{b^2-2n} - \frac{1}{b^2-2n+2}\rp\\
  a^{(2)}_0(n)\,\hnu_n
  &=
    \frac{n (b^2-2n-2)^2 +
    2b^2(b^2-2n-1)^2+(n+1)(b^2-2n)^2}{2(b^2-2n)(b^2-2n-2)} \\
  &= b^2+ n+\frac12 - \frac{2n}{b^2-2n} + \frac{2(n+1)}{b^2-2n-2}\\
  a^{(2)}_1(n) &=  2b \lp 1+ \frac{1}{b^2-2n-2} + \frac{1}{b^2-2n-4}\rp\\
  a^{(2)}_2(n)&= 1 + \frac{4}{b^2-2n+4}.
\end{align*}
The above expressions match the forms for the
$a_k(n;b)$ shown in \eqref{eq:rr5}.
\end{proof}

% Let $\fS_n$ denote the shift operator in the variable $i$, so that
% \[ \fS_n^j % f(n) = f(i+j),\quad j\in \Z \] % The difference
% operators
% % \[ \Delta_k(n,\fS_n;b) = \sum_{j=-k}^k a_j(n;b) \fS^j_n,\quad %
% k=2,3\] generate a commutative algebra which is isomorphic to % the
% polynomial algebra above.  The story is almost the same as in the %
% theory of exceptional Hermite polynomials.  The difference here is %
% that we cannot directly apply a bispectral relation and the transform
% % a certain differential operator in $z$ into a difference operator in
% % $i$.  A certain generalization is required.

\section{Characteristic values}
\label{sect:charvals}
By Lemma \ref{lem:A2adj}, the
exceptional Krall polynomials $\th_{n+1}(x;b),\; n\in \Nz$ belong to
$\ker \gamma_b$, where the latter is the functional defined in
\eqref{eq:gammabdef}. By \eqref{eq:H1def} \eqref{eq:tHdef},
$\th_{n+1}(x) = (b^2-2n)x^{n+1} +$ lower degree terms. Consequently,
for generic values of $b$, that is if $b^2\notin 2\Nz$, we have
$\deg_x \th_{n+1} = n+1,\; n\in \Nz$, and consequently the
$\th_{n+1}(x;b),\; n\in \Nz$ constitute a basis of $\ker \gamma_b$, a
codimension-1 polynomial subspace.  Equivalently, the
quasi-polynomials $\hh_n(x),\; n=0,1,2,\ldots$ constitute a basis
for the quasi-polynomial vector space $\cU_b$ defined in
\eqref{eq:Ubdef}.  This is no longer true when $b^2=2j$ for some
$j\in \Nz$.

Going forward, set
\begin{equation}
  \label{eq:bjdef}
  b_j := \sqrt{2j},\; j\in \Nz.
\end{equation}
We will refer to $\{ b_j: j\in \Nz \}$ as the set of characteristic
$b$-values.  When $b$ attains one of these characteristic values, the
properties of exceptional Krall polynomials detailed above have to be
suitably amended.  The case of $b=0$ is special, in that
\[ \th_n(x;0)= -2 \Wr[h_1,h_n]. \] This means that the resulting $b=0$
family belongs to the class of exceptional Hermite polynomials and we
will not consider it.

The first deviation from the generic theory is the fact that for
characteristic $b$ values the  quasi-polynomials
$\hh_n(x),\; n\in \Nz$ no longer constitute a basis of $\cU_b$, where
the latter is the codimension 1 subspace of $(x+b)^{-1} \cP$ defined
in \eqref{eq:Ubdef}. This is due to a linear dependence of two of the
quasi-polynomials, as described below.
\begin{prop}
  \label{prop:UVdegen}
  For a fixed $j\in \N$ we have
  \begin{align}
    % \label{eq:Udegen}
    % &b_j \hh_j(x;b_j)+ j\hh_{j-1}(x;b_j) = 0\\
  \label{eq:Udegen}
    &b_j \hh_j(x;b_j) + j \hh_{j-1}(x;b_j) = 0\\
    \label{eq:Vdegen}
    &2 \hh_j(x;b_j)+ b_j\hh_{j-1}(x;b_j) = 0
  \end{align}
\end{prop}
\noindent
\noindent
Before giving the proof, let us note that since $b_j^2 = 2j$, the
above linear relations are scalar multiples of each other.  The proof
is based on the following differential identities.
\begin{lem}
  Let $\hA,U$ and $V$ be the differential operators defined in
  \eqref{eq:hAdef} \eqref{eq:Uxdef} \eqref{eq:Vxdef} and set
  \begin{align}
    \label{eq:hUxdef}
    \hU(x,\partial_x) &:= (x+b)\circ U\circ (x+b)^{-1}  = U-
                        (x+b)^{-1}\\
    \label{eq:hVxdef}
    \hV(x,\partial_x) &:= (x+b)\circ V\circ (x+b)^{-1}  = V+ (x+b)^{-1}
  \end{align}
  We then have
  \begin{align}
    \label{eq:hAU}
      \hA U &=\hU(T_2+b^2)\\
    \label{eq:hAV}
    \hA V &=\hV(T_2+b^2-2)
  \end{align}
\end{lem}
\begin{proof}
  By \eqref{eq:T2UV} \eqref{eq:UVJ},
  \begin{align*}
    (T^2+b^2)U
    &= U (T^2+b^2+2) -2b\\
    (T^2+b^2) V
    &= V (T^2+b^2-2) +2b\\
    U^2
    &= T^2+b^2+2(x+b)(U-b)\\
    UV
    &= VU+2 = -(T^2+b^2-2)+ 2 b(x+b)\intertext{Hence, by \eqref{eq:hAdef}}
      \hA U
    &= (T_2+b^2)U-(x+b)^{-1} U^2\\
    &= U(T^2+b^2+2)-2b- (x+b)^{-1}(T^2+b^2)- 2(U-b)\\
    &= \hU(T^2+b^2)\\
    \hA V
    &= (T^2+b^2)V-(x+b)^{-1}UV\\
    &= V(T^2+b^2-2)+2b + (x+b)^{-1}(T^2+b^2-2) - 2b\\
    &= \hV (T^2+b^2-2)         
  \end{align*}
  % Hence,
  %   By the definition \eqref{eq:hAdef} and by \eqref{eq:T2UV} \eqref{eq:UVJ},
  %   \begin{align*}
  %   % \hA U &= -VU^2 - (x+b)^{-1} U^2+ 2b(x+b)U\\
  %   %       &=(-UV+2)U+ (x+b)^{-1} (V-2(x+b))U+ 2bU\circ(x+b) -2b\\
  %   %       &=-UVU+ (x+b)^{-1} VU+ 2bU\circ(x+b) -2b\\
  %   %       &=U(-VU+2b (x+b))-(x+b)^{-1} (-VU+2b(x+b))\\
  %   %       &=(U-(x+b)^{-1})(T_2+b^2)\\
  %   \hA V &= -VUV - (x+b)^{-1} UV+ 2b(x+b)V\\
  %         &= -V^2U-2V - (x+b)^{-1} VU-2(x+b)^{-1}+ 2bV\circ (x+b)+2b\\
  %         &= V(-VU+2b (x+b) -2) + (x+b)^{-1}(-VU +2b(x+b) -2)\\
  %         &= (V+(x+b)^{-1})(T^2+b^2-2)
  % \end{align*}
\end{proof}
\begin{proof}[Proof of Proposition \ref{prop:UVdegen}.]
  Let $\oU(n,\fS_n)$ be the difference operator in
  \eqref{eq:UVdef}. Observe that
  \[ \hA U h_n = \hA \oU h_n= \oU \hA h_n = \oU \hh_n,\quad n\in
    \Nz.\] Thus, relation \eqref{eq:Udegen} is equivalent to the
  assertion that $\oU \hh_{j}(x;b_j) = 0$.    This follows from
  \eqref{eq:hAU} and the fact that
  \[ (T_2 + b_j^2) h_{j} = (-2j+2j) h_{j} = 0.\]
  Similarly, let $\oV$ be the difference operator defined in
  \eqref{eq:UVdef}. Observe that
  \[ \hA V h_n = \hA \oV h_n= \oV \hA h_n = \oV \hh_n,\quad n\in
    \Nz.\] Thus, relation \eqref{eq:Vdegen} is equivalent to the
  assertion that $\oV \hh_{j-1}(x;b_j) = 0$.    This follows from
  \eqref{eq:hAV} and the fact that
  \[ (T_2 + b_j^2-2) h_{j-1} = (-2j+2+2j-2) h_{j-1} = 0.\]
\end{proof}

The second deviation from the generic theory is the fact that the
recurrence relation \eqref{eq:rr5} is problematic because
$\hv_{j-1}(x;b_j), \hu_j(x;b_j)$, as given in \eqref{eq:hvdef}
\eqref{eq:hudef}, are ill-defined indefinite forms.  The obvious way
to circumvent this difficulty is to adopt the limit definitions
\begin{align}
  \label{eq:hvlim}
  \hv_{j-1}(x,b_j)
  &= \lim_{b\to b_j} \frac{2\hh_{j}(x,b)+b\hh_{j-1}(x,b)}{b^2-2j}\\ 
  \label{eq:hulim}
  \hu_j(x,b_j)
  &= \lim_{b\to b_j} \frac{b\hh_j(x,b)+j   \hh_{j-1}(x,b)}{b^2-2j}. 
\end{align}
These limits can be explicitly evaluated thanks to the following
identities.  The expressions below constitute an alternate formulation
of $\hv_n,\hu_n$ in terms of classical Hermite polynomials that are
valid for all values of $b$.
\begin{prop} Let $\hU,\hV$ be the differential operators defined in
  \eqref{eq:hUxdef} \eqref{eq:hVxdef}. We have
  \begin{align}
    \label{eq:hudef2}  \hu_n &:= \hU h_n,\\
    \label{eq:hvdef2} \hv_n  &:= \hV h_n.
\end{align}
\end{prop}
\begin{proof} Let $\oU,\oV$ be the difference operators defined in
  \eqref{eq:UVdef}. Observe that
  \begin{align*}
    U h_n(x) &= n h_{n-1}(x) + b h_n(x) = \oU h_n(x)\\
    V h_n(x) &= -n h_{n-1}(x) + 2x h_n(x) + b h_n(x) =2h_{n+1}(x)+b
               h_n(x)=\oV h_n(x) 
  \end{align*}
  Hence, by \eqref{eq:hAU} \eqref{eq:hAV},
  \begin{align*}
    \hA U h_n(x) &=  \oU \hh_n(x)  = (b^2-2n)\hU h_n(x)\\
    \hA V h_n(x) &= \oV \hh_n(x) = (b^2-2n-2)\hV h_n(x)
  \end{align*}
  Relations \eqref{eq:hudef2} \eqref{eq:hvdef2} follow immediately
  from the definitions of $\hoU, \hoV$ in \eqref{eq:UVdef}.
\end{proof}
\noindent
With these definitions of $\hu_n,\hv_n$, relation \eqref{eq:rr5}
remains valid when $b=b_j,\; j\in \N$.  However, when
$n\in \{j-2,j-1,j,j+1,j+2\}$, the LHS of \eqref{eq:rr5} can no longer
be considered to be a linear combination of eigenfunctions of $\hT_4$.
However, this is to be expected because of the following observation.

The quasi-polynomial vector space $\cU_b$ is invariant with respect to
the action of the eigenoperator $\hT_4$.  As already noted,
generically, the action of $\hT_4$ is diagonalizable.  This is no
longer the case when $b$ attains a characteristic value.  If $b=b_j$,
then the zero eigenvalue is degenerate because
\[ \lambda_{b_j}(j) = \lambda_{b_j}(j-1) = 0, \] where
\begin{equation}
\label{eq:lambndef}
\lambda_b(n) = p_2(-2n) = (b^2-2n) (b^2-2n+2).
\end{equation}
However, there is only one eigenfunction at that eigenvalue because of
\eqref{eq:Udegen}.  Rather the spectral degeneracy corresponds to the
existence of a Jordan 2-block in the $\hT_4$ action, or what is
equivalent, a generalized eigenvector of $\hT_4$.
\begin{prop}
  Fix a $j\in \N$ and let $b=b_j$. Then
  \begin{equation}
    \label{eq:T4genevec}
    \begin{aligned}
      \hT_4 \hu_j&= 4j\hh_{j-1}\\
      \hT_4^2 \hu_j&= 0
    \end{aligned}
  \end{equation}
\end{prop}
\begin{proof}
  By \eqref{eq:hTeigen} and \eqref{eq:hudef}, and for general values
  of $b$, we have
  \begin{align*}
    \hT_4 \hu_n
    &= b(b^2-2n-2) \hh_n+ n(b^2-2n+2) \hh_{n-1}
  \end{align*}
  The desired conclusion now follows by setting $n\to j, b\to b_j$ and
  using \eqref{eq:Udegen}.
\end{proof}

\section{Bispectrality}
\label{sect:bispectrality}
The construction of the exceptional Krall polynomials should not be
seen as an isolated example, but rather should be interpreted as an
instance of a bispectral Darboux transformation.  We devote the rest
of this section to an explanation of this idea.

\subsection{Review of basic concepts.}
% \begin{enumerate}
% \item Fix a merormorphic function $\psi(x,z)$ and consider the
%   equation
%   \[ T(x,\partial_x) \psi(x,z) = T^\flat(z,\partial_z) \psi(x,z).\] We
%   define $\cA_\psi$ to be the algebra of operators $T$ for which such
%   a $T^\flat$ exists.  The map $T\to T^\flat$ is an anti-isomorphism
%   from $\cA_\psi$ to $\cA_\psi^\flat:= \cA_\tpsi$ where
%   $\tpsi(x,z) = \psi(z,x)$.  Following [Casper,Yakimov] we refer to
%   $\cA_\psi,\cA_\psi^\flat$ as the left- and right- Fourier algebras
%   associated with the joint eigenfunction $\psi(x,z)$.  
% \item We wish to define commutative subalgebras $\cR, \cR^\flat,\cR^\natural$
%   corresponding to bispectral operators.  Only the $D\Delta$
%   bispectrality requires a slightly different notion of $\cR^\flat$,
%   namely relations of the form
%   \[ \hT(x,\partial_x) p(x) \hPsi = Q(z,\partial_z) \hPsi(x,z)\] where
%   $\hT^\flat$ is a polynomial in $z\partial_z$; i.e. $\hT$ is an
%   eigenoperator for the coefficients $\phi(x)$ of $\hPsi(x,z)$. This
%   allows us to multiply $Q^\natural$ by the reciprocal of the
%   eigenvalues to obtain a recurrence relation.
% \item 
% \end{enumerate}

Given operators $L(x,\partial_x)$ and $Q(z,\partial_z)$ and a
bivariate function $\psi(x,z)$ we say that $(L,Q,\psi)$ is a
bispectral triple if there exist non-constant eigenvalue functions
$f(z), g(x)$ such that
\begin{align}
  \label{eq:Lfpsi}
  L(x,\partial_x) \psi(x,z) &= f(z) \psi(x,z) \\
  \label{eq:Qgpsi}
  Q(z,\partial_z) \psi(x,z) &= g(x) \psi(x,z).
\end{align}
This notion was introduced by Duistermaat and Gr\"unbaum in \cite{DG87}
who classified all bispectral Schr\"odinger operators. In \cite{W93},
Wilson considered the more general question of bispectral algebras of
differential operators and classified all maximal bispectral algebras
of rank one by introducing the idea of a bispectral involution.

Wilson's involution leads naturally to the notion of an
anti-homomorphism $\flat$ from the algebra of differential operators
in $x$ to the algebra of differential operators in $z$ with the
property that
\[ L(x,\partial_x) \psi(x,z) = L^\flat(z,\partial_z) \psi(x,z).\] This
idea occurs in \cite{BHY96,KM20,KR97} where it is referred to as a
bispectral anti-homomorphism of stabilizer algebras, and evolved into
the notion of a left- and right- Fourier algebra that was carefully
studied in \cite{CY20}.

The bispectrality notion also applies to heterogeneous triples where
the dual variable is discrete.  Given a differential operator
$T(x,\partial_x)$, a difference operator $\oX(n,\fS_n)$ and a sequence
of functions $\phi_n(x)$ we say that $(T,\oX,\phi)$ is a a CD
(continuous-discrete\footnote{This terminology was introduced in
  \cite{GH97}.}) bispectral triple if there exists an eigenvalue
sequence $\lambda_n$ and an eigenvalue function $q(x)$ such that
\begin{equation}
  \label{eq:LDelpsi}
\begin{aligned}
  T\phi_n &= \lambda_n \phi_n\\
  \oX \phi_n &= q \phi_n.
\end{aligned}
\end{equation}
There is a well-defined way of obtaining a CD triple
\eqref{eq:LDelpsi} from certain differential relations by treating the
joint eigenfunction $\psi(x,z)$ as a generating function whose
coefficients are the joint CD eigenfunctions $\phi_n(x)$.  This
approach is fully detailed in \cite{KM20,KM25}, but we summarize the
relevant details here.

Begin by writing
\begin{equation}
  \label{eq:psixphin}
   \psi(x,z) = \sum_{n=0}^\infty \phi_n(x) \frac{z^n}{n!}. 
\end{equation}
% We then have
% \[ z \partial_z \psi(x,z) = \sum_{n=0}^\infty n \phi_n(x)
%   \frac{z^n}{n!}. \]
We define a morphism $\natural$ from the algebra of differential
operators with polynomial coefficients to the algebra of difference
operators with polynomial coefficients by considering the action of
$z,\partial_z$ on generating functions.  In other words, for an
operator $X(z,\partial_z)$ we wish $X^\natural(n,\fS_n)$ to be a
difference operator such that
\[ X \psi(x,z) = \sum_{n=0}^\infty X^\natural \phi_n(x) \frac{z^n}{n!}.\]
Observe that
\begin{align*}
  z \psi(x,z)
  &= \sum_{n=0}^\infty \phi_n(x) \frac{z^{n+1}}{n!}
    = \sum_{n=0}^\infty n \phi_{n-1}(x) \frac{z^n}{n!}\\
  \partial_z \psi(x,z)
  &= \sum_{n=1}^\infty \phi_n(x) \frac{z^{n-1}}{(n-1)!}
    = \sum_{n=0} \phi_{n+1}(x) \frac{z^n}{n!}
\end{align*}
The homomorphism in question is therefore generated by
\[ \natural: z \mapsto n \fS_n^{-1},\quad \partial_z \mapsto \fS_n.\]

Classical Hermite polynomials serve as an illustration of these
principles. Let
\begin{equation}
  \label{eq:tJdef}
  J(z,\partial_z) := \partial_z+z/2
\end{equation}
\noindent
On the one hand, the generating function $\Psi(x,z)$
\eqref{eq:Psidef} is the joint eigenfunction of the bispectral triple
$(\partial_x, J, \Psi)$ with
\begin{align}
  \label{eq:DxPsi0z}
  \partial_x \Psi &= z \Psi\\
  \label{eq:DzPsi0x}
  J \Psi &= x \Psi.
\end{align}
Consequently, the bispectral anti-homomorphism associated to
$\Psi(x,z)$ is generated by
\[ \flat: x \mapsto J(z,\partial_z),\quad \partial_x \mapsto z.\]
On the other hand, the polynomial sequence $h_n(x)$, is the joint
eigenfunction of the following CD eigenrelations:
\begin{equation}
  \label{eq:hermDDel}
\begin{aligned}
  T_2 h_n &= -2n h_n\\
  \oJ h_n &= x h_n,
\end{aligned}
\end{equation}
where $T_2(x,\partial_x)$ is the Hermite eigenoperator
\eqref{eq:T2def} and $\oJ=J^\natural$ is the Hermite Jacobi operator
\eqref{eq:Jdef}.  Thus, the CD Hermite triple consists of the
eigenvalue relation \eqref{eq:hermde}, and of the 3-term recurrence
relation \eqref{eq:hermrel3}.  Observe that
\begin{gather*}
  T_2^\flat = z^2-2 z\circ (\partial_z+z/2) =
  -2z \partial_z
  % T_2^{\flat\natural} = -2n\\
  % x^\flat = \tX,\\
  % x^{\flat\natural} =\oJ.
\end{gather*}
Thus, the Hermite CD triple \eqref{eq:hermDDel} corresponds to
the following differential relations
\begin{equation}
  \begin{aligned}
    T_2(x,\partial_x) \Psi(x,z) &= -2z \partial_z
    \Psi(x,z)\\
    J(z,\partial_z) \Psi(x,z) &= x \Psi(x,z)
  \end{aligned}
\end{equation}

\subsection{Bispectral Darboux transformations.}
% We now indicate how the exceptional Krall Hermite polynomials may be
% constructed with the the framework of a bispectral Darboux
% transformation.  This notion was introduce in [Kasman,Rothstein] and
% [Horozov,Yakimov] who used it to construct examples of rank 2
% bispectral triples.
We now indicate how the exceptional Krall Hermite polynomials may be
understood within the framework of higher rank bispectral Darboux
transformations. As was mentioned in the introduction, this
construction was first considered in \cite{GHH97}.  However, this
earlier work did not focus on orthogonality and only considered a
single bispectal triple consisting of a 4th order differential
operator and a 5th order recurrence relation.  However, the general
philosophy of bispectral Darboux transformation requires one to study
bispectral commutative algebras.  For the case of the
quasi-polynomials $\hh_n(x),\; n\in \Nz$ considered here, this
consideration requires us to also exhibit a 6th order eigenoperator
and a 7th order recurrence relation.

Let $\hA(x,\partial_x)$ be the
intertwining operator \eqref{eq:hAdef} that is used to dress the
classical Hermite eigenoperator $T_2(x,\partial_x)$ \eqref{eq:T2def}
to obtain the Krall eigenoperator $\hT_4(x,\partial_x)$
\eqref{eq:hT4}.  As a consequence of this relation and of
\eqref{eq:pT02}, we have the intertwining relation
\[ \hT_4 \hA = \hA\, p_2(T_2); \] which serves as the Darboux
transformation in question.  We adapt the techniques in \cite{KR97} to
explain why this construction preserves the bispectrality of the
classical Hermite eigenrelation.  In other words, we show how to
``dress'' the classical recurrence relations to obtain the higher
order recurrence relations enjoyed by the exceptional Krall
polynomials.

First, we note that the construction of the operator $\hT_4$ is an
example of a ``rank 2'' Darboux transformation in that the operators
in the centralizer of $T_2$ all have even degree and the construction
utilizes that operator's $2$-dimensional eigenspace
\cite{BHY96a,K97,LP94}. One can see that $T_2$ does not commute with
any operators of odd order since its image in the Weyl algebra under
the anti-isomorphism $x\mapsto\partial_x$ and $\partial_x\mapsto x$ is
first order and hence only commutes with polynomials in itself.

Now let
\begin{align*}
  f_1(x;a) &:= \cos\lp \frac{\pi a}2\rp \Gamma\lp \frac{a}2+\frac12\rp
             M\!\lp\! -\frac{a}2,\frac12,x^2\rp
             +2 \sin\lp \frac{\pi a}2\rp\Gamma\lp \frac{a}2+1\rp
             x M\!\lp\!\frac12 -\frac{a}2,\frac32,x^2\rp\\
  f_2(x;a) &:= \sin\lp \frac{\pi a}2\rp \Gamma\lp \frac{a}2+\frac12\rp
             M\!\lp\! -\frac{a}2,\frac12,x^2\rp
             -2 \cos\lp \frac{\pi a}2\rp\Gamma\lp \frac{a}2+1\rp
             x M\!\lp\!\frac12 -\frac{a}2,\frac32,x^2\rp
\end{align*}
where
\[ M(a,c,z)=\,_1F_1(a,c,z) = \sum_{n=0}^\infty \frac{ (a)_n}{(c)_n}
  \frac{z^n}{n!},\] is the usual confluent hypergeometric function.
The function $y(z) = M(a,c,z)$ is a solution of Kummer differential equation
\[ z y''(z)+ (c-z) y' + a y = 0.\] Consequently, $f_1(x;a), f_2(x;a)$
constitute a basis for $\ker(T_2 +2a)$; i.e.
\begin{equation}
  \label{eq:T2af}
   T_2(x,\partial_x) f_k(x;a) = -2a f_k(x;a),\quad k=1,2.
\end{equation}
Indeed, for $a\in \Nz$, the first solution reduces to a multiple of
Hermite polynomials:
\[ f_1(x;n) = \sqrt\pi h_n(x),\quad n\in \Nz.\]

Both functions satisfy the 3-term recurrence relation
\eqref{eq:hermrel3} enjoyed by Hermite polynomials:
\begin{equation}
  \label{eq:fRR}
 x f_k(x;a)= f_k(x;a+1)+ \frac{a}{2} f_k(x;a-1),\quad k=1,2.  
\end{equation}
This is a consequence of the following elementary identities:
\begin{align}
  &(a-c+1) M(a,c,z) - a M(a+1,c,z) + (c-1) M(a,c-1,z) = 0,\\
  &(c-a) M(a-1,c,z) + (2a-c+z) M(a,c,z)-a M(a+1,c,z)=0.
\end{align}

Further let
\begin{equation}
  \label{eq:cbdef}
  \begin{aligned}
    c_b &= \ev_{b^2/2}(n) \circ \oU(n,\fS_n)\\
    c_b(f(a))&=2f(b^2/2)+bf(b^2/2-1)    
  \end{aligned}
\end{equation}
be the linear functional that evaluates its argument at the indicated
points\tre{\footnote{Hey, that's weird.  The distribution is not
    supported at one point, but at two!  Maybe I'm not thinking about
    it right, but that seems to be one rather unusual thing about this
    construction in the context of bispectral Darboux transformations.
    Perhaps that should also be commented upon in the closing
    remarks?}} and takes this linear combination.  As a consequence of
\eqref{eq:T2af} and \eqref{eq:fRR}, we have
\[
  \hA(x,\partial_x) f_k(x;a) = (b^2-2a)f_k(x;a) -\frac{b f_k(x;a)+ a
    f_k(x;a-1)}{x+b},\quad k=1,2,
\]
where $\hA(x,\partial_x)$ is the
intertwiner defined in \eqref{eq:hAdef}.  
It follows immediately that
\[ \hA(x,\partial_x) c_b(f_k(x,a)) = 0,\quad k=1,2 .\] Thus, $\hA$ is
indeed the unique monic ordinary differential operator whose kernel is
spanned by $c_b(f_1(x,a))$ and $c_b(f_2(x,a))$.  By \eqref{eq:T2af},
\[ p_2(T_2) f_k =  (b^2-2a)(b^2-2-2a) f_k,\quad k=1,2.\]
Hence,
\[ p_2(T_2) c_b(f_k) = c_b(p_2(T_2) f_k) =  0 ,\quad k=1,2.\]
%It then follows from the fact that
%$c_b(p_2(z)p(z))=0$ for any polynomial $p(z)$ that
%$$
%\hA\circ p_2(T_2)(c_b(f_n))=\hA(c_b(p_2(z)f_n)=0.
%$$
It follows that $\ker(\hA)\subset\ker(\hA p_2(T_2))$, and  hence there
exists a unique differential operator $\hT_4$ satisfying the
intertwining relationship $ \hT_4 \hA = \hA\, p_2(T_2)$.  This
establishes the construction as a rank 2 Darboux transformation.

The next remark is that a bispectral Darboux transformation involves
not just two operators, but rather two commutative algebras of
bispectral operators.  To that end, define $R_{\hA,T_2}$ to be the
algebra of polynomials $p(T)$ with the property that there exists a
$\hT_p(x,\partial_x)$ such that
  % \footnote{The factor of $-1/2$ is there for
  %   convenience.  It allows us
  %   to write $p(n)$ in a subsequent relation.}
\begin{equation}
  \label{eq:hApT2}
  \hA p(T_2) = \hT_p \hA.
\end{equation}
We define
\begin{equation}
  \label{eq:cRdef}
  \cR = \{ \hT_p : p\in R_{\hA,T_2} \}.
\end{equation}
to be the algebra of all such operators. By construction, $\cR$ is
isomorphic to $R_{\hA,T_2}$ and is therefore a commutative algebra.

In this case, the polynomial algebra $R_{\hA,T_2}$ consists of
polynomials that vanish at $-b^2$ and $2-b^2$, while the operator
algebra $\cR$ consists of differential operators that act as
``eigenoperators'' of the quasi-polynomials $\hh_n(x)$.  Since
$T_2 h_n = -2n h_n$, relation \eqref{eq:hApT2} is equivalent to
\[ \hT_p \hh_n = p(-2n) \hh_n,\quad p\in R. \] The fourth-order
eigenoperator $\hT_4$ defined in \eqref{eq:hT4} is precisely
$\hT_{p_2}$. The general element of $R_{\hA,T_2}$ is a
polynomial of the form
$p(T) = \tp(T) p_2(T)$ where $\tp(T)$ is an arbitrary polynomial.
Thus, 
the general element of $\cR$ is an operator of the form
\[ \hT_p = \hA_2 \tp(T_2) \hA_2^\dag. \]
It follows that, as an algebra, $\cR$ is generated by $\hT_4$ and by
the 6th order eigenoperator
\[ \hT_6 =  \hA_2 T_2 \hA_2^\dag. \]
The corresponding eigenrelation is
\[ \hT_6 \hh_n = -2n(b^2-2n)(b^2-2n-2) \hh_n,\quad n=0,1,2,\ldots.\]

% Set
% \begin{equation}
%   \label{eq:tA2B2def}
%   \begin{aligned}
%     \tA_2(z,\partial_z) &:= ((x+b) \hA)^\flat 
%     = \lp (x+b)(T_2 + b^2) -(\partial_x+b)\rp ^\flat\\
%     &=(b^2-z\partial_z)\circ (\partial_z+z/2+b)-(z+b)\\
%     \tA_2^\natural(n,\fS_n)
%     &= (b^2-2n)(\oJ+b) - \oU\\
%     \tB_2  &:=\hA ^\dag \circ (x+b)= (T_2 + b^2-2)\circ (x+b) +
%     \partial_x + b. 
%   \end{aligned}
% \end{equation}
% Consequently,
% \begin{align*}
%   \hPsi(x,z) &= (x+b)^{-1} \tA_2(z,\partial_z) \Psi(x,z),\\
%   \hh_n(x) &= (x+b)^{-1} \tA_2^\natural(n,\fS_n) h_n(x).
% \end{align*}

The intertwiner $\hA$ has rational coefficients and therefore does not
admit a dual $\hA^\flat$.  Thus, as per relation (1.8) of \cite{KR97}
\tre{\footnote{Warning: our use of the tilde decoration differs from that
  paper.}} the construction of the bispectral duals requires the use of
the normalized intertwiner and its images under the bispectral
anti-homomorphism:
\begin{align}
  \label{eq:tAdef}
  \tA &= (x+b)\hA\\
  \label{eq:tA2flat}
  \tA^\flat(z,\partial_z)
      & = \lp (x+b)(T_2 + b^2) -(\partial_x+b)\rp ^\flat\\ \noindent
      &=(b^2-2z\partial_z)\circ (\partial_z+z/2+b)-(z+b)\\
  \label{eq:tAnatural}
  \tA^{\flat\natural}(n,\fS_n)
      &= (b^2-2n)(\oJ+b) - \oU
\end{align}
We now define $Q=R_{\tA^{\flat\natural}, \oJ}$ to be the algebra of
polynomials $q(x)$ for which there exists
% an operator
% $\hX_q(z,\partial_z)$ such that
% \[  \tA^\flat q(x^\flat) = \hX_q \tA^\flat.\]
% Equivalently, and perhaps somewhat more usefully, $Q$  consists of
% polynomials $q(x)$  for which there exists
a difference operator
$\hoX_q(n,\fS_n)$ such that
\begin{equation}
  \label{eq:hoXqtA}
   \hoX_q \circ \tA^{\flat\natural} = \tA^{\flat\natural} \circ q(\oJ).
\end{equation}
Thus, we regard $\hoX_q$ as the dressing of the classical operator
$q(\oJ)= q(x^\flat)^{\natural}$ by the discrete intertwiner
$\tA^{\flat\natural}$.

Since
\[ \hh_n = (x+b)^{-1} \tA^{\flat\natural} h_n, \]
the intertwining relation \eqref{eq:hoXqtA} is equivalent to
\begin{equation}
  \label{eq:qhhX}
  q \hh_n = \hoX_q \hh_n,
\end{equation}
which represents a higher-order recurrence relation with multiplier $q(x)$.

\begin{prop}
  We claim that $q\in Q$ if and only if the operator
  \begin{equation}
    \label{eq:Xqdef}
     X_q:=\hA^\dag q \hA.
  \end{equation}
has polynomial coefficients.
\end{prop}
\begin{proof}
  Suppose that \eqref{eq:Xqdef} holds.  Set
  \begin{equation}
    \label{eq:hoXqdef}
    \hoX_q:= X_q^{\flat\natural}\circ p_2(-2n)^{-1}.
  \end{equation}
  where $p_2(T)=(T+b^2)(T+b^2-2)$ is the eigenvalue polynomial 
  \eqref{eq:p2def}. By \eqref{eq:hT4} \eqref{eq:hTeigen}, it
  follows that
  \[ \hA \hA^\dag q \hA h_n = \hT_4 q \hh_n = X_q^{\flat\natural}
    \hh_n = \hoX_q p_2(-2n) \hh_n = \hT_4 \hoX_q \hh_n.\]
  For generic values\footnote{The operator $\hoX_b$ is not
    well-defined for non-characteristic values.  For such cases one
    must take the approach outlined in Section \ref{sect:charvals}.} of
  $b$  the operator $\hT_4$ does not annihilate any elements of
  $\cU_b$. Therefore, \eqref{eq:qhhX} holds.

  Conversely, suppose that \eqref{eq:qhhX} holds.  It follows that
  \[ \hT_4 \hoX_q \hh_n = \hoX_q p_2(-2n) \hh_n  = \hA \hA^\dag  q
    \hA h_n,\quad n\in \Nz.\]
  Since $\hT_4q \hh_n\in \cU_b$ for every $n\in \Nz$, it follows that
  $\hA^\dag q \hA h_n\in \cP$ for every $n$.  This implies that
  $\hA^\dag q \hA$ maps $\cP$ to $\cP$, and therefore must have
  polynomial coefficients.
\end{proof}
Next, we define $\cQ$ to be the algebra of operators
$\hoX_q,\; q\in Q$ as defined in \eqref{eq:hoXqdef}. By the above
Proposition, $\cQ$ is isomorphic to $Q$ and hence is a commutative
algebra.
Effectively, $Q$ and $\cQ$ are, respectively, the
isomorphic commutative algebras of multipliers and of difference
operators that arise in the recurrence relations enjoyed by the 
quasi-polynomials $\hh_n(x)$.

By inspection of \eqref{eq:hAdef} and \eqref{eq:hAadjdef}, the
polynomial algebra $Q$ consists of polynomials $q(x)$ with a double
root at $x=-b$.  Thus, as an algebra $\cQ$ is generated by difference
operators $\hoX_{q_2}, \hoX_{q_3}$ where
$q_k(x) = (x+b)^k,\; k\in \Nz$.  These correspond to, respectively, a
5th order and a 7th order recurrence relation for the
quasi-polynomials $\hh_n(x)$.  In the remainder of this section we
give a bispectral proof of the 5th order recurrence relation
\eqref{eq:rr5} and then exhibit and prove the 7th order recurrence
relation corresponding to multiplication by $(x+b)^3$.

\begin{proof}[Bispectral proof of Proposition \ref{prop:rr5}]
  Let $T_2,\hA, U, V,\hU, \hV$ be the differential operators defined
  in \eqref{eq:T2def} \eqref{eq:hAdef} \eqref{eq:Uxdef}
  \eqref{eq:Vxdef} \eqref{eq:hUxdef} \eqref{eq:hVxdef}.  By
  \eqref{eq:Uconj} \eqref{eq:Vconj} \eqref{eq:T2conj} and direct
  calculation,
  \begin{equation}
    \label{eq:rr5idents}
  \begin{aligned}
    [ x+b, T_2] &=  V-U\\
    [x+b,\hA]  &= V-U+(x+b)^{-1}\\
    [(x+b)^2,\hA]  &= 2(x+b)(V-U)=V^2-U^2+2\\
    \hV V- \hU U
                    &= V^2-U^2 + (x+b)^{-1}(V+U) = V^2-U^2+2.
  \end{aligned}
  \end{equation}
  Hence, by \eqref{eq:hermDDel} \eqref{eq:hvdef} \eqref{eq:hudef}
  \eqref{eq:hudef2} \eqref{eq:hvdef2},
  \begin{align*}
    (x+b)^2 \hh_n &= (\oJ+b)^2 \hh_n + (V^{2} -
                    U^{2} + 2)h_n\\
                  &= (\oJ+b)^2 \hh_n + (\hV V-\hU U) h_n\\
    &= (\oJ+b)^2 \hh_n + \oV \hv_n - \oU \hu_n\\
    &= (\oJ+b)^2 \hh_n + (\oV \hoV - \oU \hoU) \hh_n.
  \end{align*}
\end{proof}
\noindent
\tre{One could also do the proof using the intertwining relation
\eqref{eq:hoXqtA}.  Since
\[ \hoX_{q_2} = (\oJ+b)^2 + \oV \hoV - \oU \hoU, \]
the above proof amounts to a proof of the identity
\begin{align}
  \label{eq:RR5i}
  &((\oJ+b)^2 + \oV \hoV - \oU \hoU)(b^2-2n)(\oJ+b)
  -    ((\oJ+b)^2 + \oV \hoV - \oU \hoU) \oU\\ \nonumber
  &\qquad = ((b^2-2n)(\oJ+b) -\oU) (\oJ+b)^2  
\end{align}
This can be done by applying both sides to $h_n$ and rewriting the
actions as differential operators in $x,\partial_x$.  This amounts to
applying the inverse of the anti-homomorphism $\flat\natural$.
Observe that
\begin{align*}
  &(x+b)^3 \hA = (x+b)^3 T_2 - (x+b)^2 U\\
  &((x+b)^3 \hA)^{\flat\natural} = ((b^2-2n)(\oJ+b) -\oU) (\oJ+b)^2
  \\
  &\lp (x+b)T_2((x+b)^2 + V^2 -  U^2+2)\rp ^{\flat\natural}
  = ((\oJ+b)^2 + \oV \hoV - \oU \hoU)(b^2-2n)(\oJ+b)\\
  &\lp U((x+b)^2+V^2-U^2+2)\rp^{\flat\natural}
  =((\oJ+b)^2 + \oV \hoV - \oU \hoU) \oU\\
  &(x+b)T_2((x+b)^2 + V^2 -  U^2+2)-  U((x+b)^2+V^2-U^2+2)\\
  &\qquad = (x+b)\hA((x+b)^2+V^2-U^2+2)
\end{align*}
Thus, \eqref{eq:RR5i} is equivalent to the identity
\begin{align*}
(x+b)\hA((x+b)^2+V^2-U^2+2) = (x+b)^3 \hA
\end{align*}
Thus, the bispectral proof given above is more or less the proof of
this relation.}

\begin{prop}
    \label{prop:rr7}
    The  quasi-polynomials $\hh_n(x),\; n\in \Nz$ satisfy the
    following 7th order recurrence relation.
  \begin{equation}
    \label{eq:rr7}
    \begin{aligned}
    (x+b)^3 \hh_n &= (\oJ+b)^3\hh_n+\frac32 (\oJ+b)(\oV\hoV - \oU
    \hoU) \hh_n+ \frac32 (\hoV+ \hoU) \hh_n.
    \end{aligned}
  \end{equation}
\end{prop}
\begin{proof}
  By \eqref{eq:rr5idents},
  \begin{align*}
    (x+b)^3 \hA = \hA\circ (x+b)^3+ \frac32 \lp \hV V - \hU U\rp \circ
    (x+b)+ \frac32 (\hV+ \hU)
  \end{align*}
  The desired relation follows by applying both sides to $h_n$.
\end{proof}

\end{document}